\documentclass[12pt, a4paper]{amsart}

\usepackage[hmargin=30mm, vmargin=25mm, includefoot, twoside]{geometry}
\usepackage[bookmarksopen=true]{hyperref}

\usepackage{amsfonts,amssymb,verbatim}
\usepackage{latexsym}
\usepackage{mathrsfs}
\usepackage{stmaryrd}
\usepackage{xspace}
\usepackage{enumerate, paralist}
\usepackage{graphicx}
\usepackage[all]{xy}
\usepackage{extarrows}

\usepackage[usenames,dvipsnames]{color}

\usepackage{txfonts, pxfonts}

\usepackage{amsthm}
\usepackage{amsmath}

\newtheorem{thm}{Theorem}[section]
 \newtheorem{cor}[thm]{Corollary}
 \newtheorem{lem}[thm]{Lemma}
 \newtheorem{prop}[thm]{Proposition}

\numberwithin{equation}{section}

 \theoremstyle{definition}
  \newtheorem{defn}[thm]{Definition}
  \newtheorem{question}[thm]{Question}
 \theoremstyle{remark}
 \newtheorem{rem}[thm]{Remark}
  \newtheorem{ex}[thm]{Example}

\def\lp{\ell^p_E(X)}
\def\lo{\ell^1_E(X)}
\def\lz{\ell^0_E(X)}
\def\linf{\ell^\infty_E(X)}

\def\Blp{\mathfrak{B}(\ell^p_E(X))}
\def\B{\mathfrak B}

\def\Dom{\{0\} \cup [1,\infty]}
\def\dom{\{0\} \cup [1,\infty)}

\def\VL{\mathrm{VL}(X)}
\def\VLinf{\mathrm{VL}_\infty(X)}

\def\Id{\mathrm{Id}}
\def\Roe{\mathrm{Roe}(X,B)}
\def\K{\mathcal{K}(X,B)}

\def\supp{\mathrm {supp}}
\def\andx{\quad\mbox{~and~}\quad}

\def\N{\mathcal N}

\def\Commut{\mathrm{Commut}}

\def\diam{\mathrm{diam}}

\begin{document}

\title[Quasi-Locality]{Quasi-Locality and Property A}

\author{J\'{a}n \v{S}pakula and Jiawen Zhang}
\address{School of Mathematics, University of Southampton, Highfield, SO17 1BJ, United Kingdom.}
\email{jan.spakula@soton.ac.uk, jiawen.zhang@soton.ac.uk}

\date{}
\keywords{Quasi-local operators, $\ell^p$-Roe-like algebras, Property A}

\thanks{J\v{S} was supported by Marie Curie FP7-PEOPLE-2013-CIG Coarse Analysis (631945).
JZ was supported by the Sino-British Trust Fellowship by Royal Society.}


\begin{abstract}
    Let $X$ be a metric space with bounded geometry, $p\in\Dom$, and let $E$ be a Banach space. The main result of this paper is that either if $X$ has Yu's Property A and $p\in(1,\infty)$, or without any condition on $X$ when $p\in\{0,1,\infty\}$, then quasi-local operators on $\ell^p(X,E)$ belong to (the appropriate variant of) Roe algebra of $X$.
    This generalises the existing results of this type by Lange and Rabinovich, Engel, Tikuisis and the first author, and Li, Wang and the second author. As consequences, we obtain that uniform $\ell^p$-Roe algebras (of spaces with Property A) are closed under taking inverses, and another condition characterising Property A, akin to operator norm localisation for quasi-local operators.
\end{abstract}
\date{\today}
\maketitle

\parskip 4pt

\noindent\textit{Mathematics Subject Classification} (2010): 20F65, 46H35, 46J40, 47L10.\\

\section{Introduction}\label{sec:intro}

Roe algebras are C*-algebras associated to metric spaces (or more generally spaces endowed with a coarse structure), which encode their large scale structures. The interest to study these objects stems from two sources: index theory on open manifolds as well as the Novikov conjecture, and operator theory of band-dominated operators. Recently, much interest has arisen in investigating their $\ell^p$-variants as well (see \cite{chung2018rigidity,LWZ2018}).

On the index theory side, the $K$-theory of Roe algebras serves as a receptacle for indices of (for instance) geometric differential operators; this is how they originally appeared in the pioneering work of John Roe \cite{roe1988:index-thm-on-open-mfds,roe1996:index-theory-CBMS}. This landscape, framed by various versions of the coarse Baum--Connes conjecture, has shown its power to produce significant results in geometry and topology, for instance for the Novikov conjecture and the bounded Borel rigidity conjecture (see e.g.~\cite{guent-tess-yu2012:Borel-stuff,yu1998:Novikov-for-FAD,yu2000:CBC-for-CE}), and about the non-existence of metrics of positive scalar curvature on open Riemannian manifolds (see e.g.~\cite{hanke-pape-schick2015,schick2014:ICM,yu1997:0-in-the-spec-PSC}).

On the single operator theory side, the elements of Roe algebras are referred to as \emph{band-dominated} operators, and substantial work has been done for instance on their Fredholmness and essential spectrum (see e.g.~\cite{lindner2006:inf-matr-book,rabinovich2012limit} and references therein). More recently, the coarse geometric ideas have been recognised as useful in this area as well (see e.g.~\cite{hagger-lindner-seidel2016,lindner-seidel2014:core-issue,seidel2014:survey,vspakula2017metric}). These ideas have been applied for instance to Hamiltonians of quantum systems and Schr\"{o}dinger operators \cite{georgescu2011:ess-spec-on-metric-spaces,georgescu2017:ess-spec-ell-diffops,georgescu-iftimovici2006}.

Roe algebras have various versions in studying different problems (for instance, uniform or stable versions). To work in the full generality, we choose the setup of $\ell^p$-Roe-like algebras, which is, roughly speaking, defined as the norm closure of some algebra of operators with \emph{finite propagation} (also known as \emph{band operators} in the operator theory language). This raises the question whether we can recognise exactly which operators belong to an $\ell^p$-Roe-like algebra, rather than producing a sequence of finite propagation approximants --- in some situations, for example when considering index theory of pseudo-differential operators (see e.g. \cite{engel2015index,engel2015rough}), this may not always be possible.

To outline the main result of this paper, let us explain the notions in the case when $(X,d)$ is a metric space with bounded geometry (for precise definitions see Section \ref{sec:preliminaries}), $p=2$, and the Banach space $E$ appearing later is just $\mathbb{C}$. Thinking of operators on $\ell^2(X)$ as $X$-by-$X$ matrices, we say that such an operator has \emph{finite propagation} (or is a \emph{band operator}), if the non-zero entries appear only in a band of finite width (measured by the metric on $X$) around the main diagonal. (See Definition \ref{def:propagation} for full details.) The finite propagation operators form a *-subalgebra of $\B(\ell^2(X))$, and its closure is called the \emph{uniform Roe algebra} of $X$. Any finite propagation operator $a$ has the property that there exists $R\geq 0$, such that for any $A,B\subset X$ with $d(A,B)\geq R$, we have $\chi_Aa\chi_B=0$ (here we denote by $\chi_A$ the characteristic function of $A$, acting as a multiplication operator on $\ell^2(X)$).
One possible weakening of this condition is to ask that for any $\varepsilon>0$ there exists $R\geq0$, such that for any $A,B\subset X$ with $d(A,B)\geq R$ one has $\|\chi_Aa\chi_B\|<\varepsilon$. Operators $a$ satisfying this condition are called \emph{quasi-local}; and it is immediate (by an approximation argument), that any operator in the uniform Roe algebra of $X$ is indeed quasi-local.

The converse, namely the question whether any quasi-local operator belongs to the corresponding Roe algebra, is the main question this paper addresses. Note that a positive answer provides a method for checking whether an operator belongs to a Roe-like algebra merely by estimating norms of off-diagonal block restrictions of the operator (i.e. $\chi_Aa\chi_B$) and does not require producing finite propagation approximants. This is crucial, for instance, for the work of Engel \cite{engel2015index,engel2015rough} on the index theory of pseudo-differential operators. To expand on this point, Engel \cite[Section 2]{engel2015index} points out that while indices of genuine differential operators on Riemannian manifolds live in the (appropriate) Roe algebra, the indices of uniform pseudo-differential operators are only known to be quasi-local.

Furthermore, a positive answer to this question is also used in the work of White and Willett \cite{white-willett2018:cartans} on classifying Cartan subalgebras of Roe algebras, with some applications on the associated rigidity problem as well.

The history of this question can be traced back to the work of Lange and Rabinovich \cite{lange1985noethericity}, where an affirmative answer is obtained in the case $X=\mathbb{Z}^N$ using the Fourier transform. (This fact has been then used extensively in the band-dominated operator theory for studying Fredholmness, see e.g.~\cite{rabinovich2012limit} for an overview.) Subsequently, it appears in Roe \cite[Remark on page 20]{roe1996:index-theory-CBMS}. The recent progress in $p=2$ case includes the result of Engel \cite{engel2015rough} for the case when $X$ is a manifold of bounded geometry with polynomial volume growth, and the result of Tikuisis and the first author \cite{spakula2017relative}, where the authors assume that $X$ has straight finite decomposition complexity (but note that $X$ is only required to be proper, not necessarily of bounded geometry). This work has been generalised by Li, Wang and the second author to general $p\in[1,\infty)$ \cite{LWZ2018} under the same assumption on the space $X$.

The main result of this paper provides an affirmative answer to this question for any metric space $X$ with bounded geometry in the following cases: without any further assumption on $X$ for $p\in \{0,1,\infty\}$; or when $X$ has Property A and $p\in(1,\infty)$.
For the former case, our proof is inspired by the recent work on limit operator theory \cite{zhang2018}. For the latter case, we utilise the key technical trick from \cite{LWZ2018,spakula2017relative} (cf.~Proposition \ref{estimate for commut}), but the remainder of our argument is of a very different nature than that of \cite{LWZ2018,spakula2017relative}: We do not inductively chop a quasi-local operator (as was done using straight finite decomposition complexity) to eventually produce an approximant of finite propagation. Instead, we prove and use a property similar to the operator norm localisation property \cite{chen2008metric}, but for quasi-local operators (cf.~Lemma \ref{ONL}).

We note that our result generalises (for metric spaces with bounded geometry) all the previous ones: both straight finite decomposition complexity and polynomial growth (separately) imply Property A. For the former, see \cite{dranishnikov2014asymptotic}; for the latter this is a result of Tu \cite{tu2001remarks}. For completeness, let us note that for finitely generated groups, polynomial growth implies finite asymptotic dimension (using Gromov's theorem \cite{gromov1981groups} one concludes that these groups are virtually nilpotent, so that a result of Bell and Dranishnikov applies \cite{bell2008asymptotic}). This in turn implies finite decomposition complexity \cite{dranishnikov2014asymptotic,guent-tess-yu2013:fdc}. However we are not aware of an argument that would assert that polynomial growth implies (straight) finite decomposition complexity for general metric spaces.

We present two direct applications of our work: We observe that the conclusion of Lemma \ref{ONL} for $p=2$ is actually equivalent to Property A (and hence to the operator norm localisation property and the metric sparsification property \cite{brodzki2013uniform,chen2008metric,sako2014property}). Thus we obtain yet another characterisation of Property A. The second application is that $\ell^p$-Roe-like algebras are closed under taking inverses (in $\B(\ell^p_E(X))$). While this is immediate when $p=2$ (the C*-algebra case), for the remaining $p$ this requires an argument (cf.~a discussion of this property in the context on band-dominated operator theory \cite{rabinovich2012limit}).

The paper is structured as follows:
After recalling the background notations, terminology and some existing results and examples in Section \ref{sec:preliminaries}, we state and discuss the main result in Section \ref{sec:mainthm}. Our proof of the main result requires treating substantial portions of the two cases ($p\in\{0,1,\infty\}$ and $p\in(1,\infty)$) differently, so it is split up along this seam into Sections \ref{sec:proof-01infty} and \ref{sec:p-in-1-infty}. The paper closes with applications and open questions in Section \ref{sec:apps-and-qs}.

\emph{Acknowledgments.} We would like to thank Kang Li and Rufus Willett for comments and discussions relating to this piece, and Yeong Chyuan Chung for several helpful comments after reading an early draft of this paper.

\section{Preliminaries}\label{sec:preliminaries}

\subsection{Notation and conventions}
For a metric space $(X,d)$, $x \in X$ and $R \geq 0$, denote by $B(x,R)$ the closed ball in $X$ with radius $R$ and centre $x$. We say that $X$ has \emph{bounded geometry}, if for any $R$, the number $\sup_{x\in X} \sharp B(x,R)$ is finite, where the notation $\sharp S$ denotes the cardinality of the set $S$.

Given $A\subset X$ and $K\geq0$ we denote the $K$-neighbourhood of $A$ by $\N_K(A)=\{x\in X\mid d(x,A)\leq K\}$. Furthermore, we shall use $\chi_A:X\to\{0,1\}$ to denote the characteristic function of the set $A$.

We shall also use the notation $C_b(X)$ for the C*-algebra of bounded continuous functions on $X$ endowed with the sup-norm; of course under the standing assumption that $X$ has bounded geometry, this is nothing else than $\ell^\infty(X)$. Nevertheless, we will use this notation to emphasise when we are considering these functions on $X$ as (soon to be) acting as multiplication operators on some Banach space. Moreover, the Roe-like algebras and several other notions we refer to make sense and are useful for ``continuous" spaces (e.g. manifolds) as well. The convenience of having a discrete space lies, for instance, in the fact that any characteristic function sits inside $C_b(X)$.

Let $E$ be a Banach space. We denote the closed unit ball of $E$ by $E_1$.
We denote the bounded linear operators on $E$ by $\B(E)$, and the compact operators on $E$ by $\mathcal{K}(E)$. Moreover, we define
$$E_\infty := \ell^\infty(\mathbb N, E)\big/ \big\{(a_n)_{n\in \mathbb N} \in \ell^\infty(\mathbb N, E): \lim_{n\rightarrow \infty}\|a_n\|=0\big\},$$
which is a Banach space, equipped with the quotient norm
$$\big\|[(a_n)_{n\in \mathbb N}]\big\|:=\limsup_{n \rightarrow \infty}\|a_n\|.$$

\emph{Throughout the paper, $(X,d)$ shall stand for a metric space with bounded geometry, and $E$ for a Banach space.}

\subsection{Banach space valued $\ell^p$-spaces and block cutdowns}
We start with the following notions of Banach space valued $\ell^p$-spaces:
\begin{itemize}
  \item $\lp:=\ell^p(X;E)$ for $p\in [1,\infty)$, which denotes the Banach space of $p$-summable functions from $X$ to $E$ with respect to the counting measure;
  \item $\linf:=\ell^\infty(X;E)$, which denotes the Banach space of bounded functions from $X$ to $E$;
  \item $\lz:=c_0(X;E)$, which denotes the Banach space of continuous functions from $X$ to $E$ vanishing at infinity.
\end{itemize}

For $p \in \Dom$, $\rho: C_b(X) \to \B(\lp)$ is a representation defined by point-wise multiplication. To simplify the notation, we write $f \xi$ instead of $\rho(f)(\xi)$ for $f \in C_b(X)$ and $\xi \in \lp$. For any $F \subseteq X$, denote $\rho(\chi_F)$ by $P_F$. Also recall a net $\{a_i\}$ converges in \emph{strong operator topology (SOT)} to $a$ in $\Blp$ if and only if $\|a_i(\xi)-a(\xi)\|_p \rightarrow 0$ for any $\xi \in \lp$. To deal with the case of $p=\infty$, we need another topology on $\B(\linf)$ as follows.
\begin{defn}
A net $\{a_i\}$ converges \emph{point-wise strongly} to $a$ in $\B(\linf)$, if for any $v \in \linf$, $a_iv$ converges point-wise to $av$ in $\linf$.
\end{defn}


We shall recall several notions from \cite{LWZ2018,spakula2017relative}, keeping the notation consistent with these sources for the convenience. The first of these is the notion of a block cutdown of an operator, whose purpose is to literally cut most of the operator away, keeping only a collection of diagonal blocks.

\begin{defn}\label{block cut down}
Given a family $(e_j)_{j \in J}$ of positive contractions in $C_b(X)$ with pairwise disjoint supports and $p \in \Dom$, define the \emph{block cutdown} map $\theta_{(e_j)_{j \in J}}: \Blp \rightarrow \Blp$ by
$$\theta_{(e_j)_{j \in J}}(a):=\sum_{j \in J}e_jae_j,$$
which converges strongly when $p \in \dom$ (by \cite[Lemma 2.4]{LWZ2018}), and point-wise strongly when $p=\infty$ (by definition). A subalgebra $B \subseteq \Blp$ is \emph{closed under block cutdowns}, if for $a \in B$, any of its block cutdowns belongs to $B$.
\end{defn}

\begin{rem}
Note that the multiplication by $C_b(X)$ commutes with the block cutdowns, i.e., for any $a \in \B(\lp)$ and $f \in C_b(X)$, we have
$$f\theta_{(e_j)_{j \in J}}(a)=\theta_{(e_j)_{j \in J}}(fa) \andx \theta_{(e_j)_{j \in J}}(a)f=\theta_{(e_j)_{j \in J}}(af).$$
Also notice that
$$\big\|\sum_{j \in J}e_jae_j\big\|=\sup_{j \in J}\|e_jae_j\|.$$
\end{rem}

\subsection{$\ell^p$-Roe-like algebras}

\begin{defn}\label{def:propagation}
Let $R \geq 0$, $p \in \Dom$ and $a \in \Blp$. Denote the support of a function $f \in C_b(X)$ by $\supp(f)$. We say:
\begin{itemize}
  \item $a$ has \emph{propagation at most $R$}, if for any $f,f' \in C_b(X)$ with $d(\supp(f),\supp(f')) > R$, then $faf'=0$.
  \item $a$ has \emph{$\varepsilon$-propagation at most $R$} for some $\varepsilon >0$,  if for any $f,f' \in C_b(X)_1$ with $d(\supp(f),\supp(f')) > R$, then $\|faf'\|< \varepsilon$.
  \item $a$ is \emph{quasi-local}, if for every $\varepsilon>0$, it has finite $\varepsilon$-propagation.
\end{itemize}
\end{defn}
Note that finite propagation operators are also called band operators in operator theory literature including \cite{rabinovich2012limit,roe2005band}. We recall the following two basic examples of operators with finite propagation:
\begin{ex}\label{mul and pt}
\begin{enumerate}
  \item Let $f: X \to \B(E)$ be a bounded function. Then the operator $A \in \Blp$ defined by $(A\xi)(x):=f(x)(\xi(x))$ for any $\xi \in \lp$ and $x\in X$, has propagation 0, called a \emph{multiplication operator}.
  \item Let $D,R$ be subsets of $X$, and $t:D \to R$ be a bijection such that $\sup_{x\in D}d(x,t(x))$ is finite. Define an operator $V$ by:
      \begin{equation*}
            (V\xi)(x)=
            \begin{cases}
              ~\xi(t^{-1}(x)), & x\in R, \\
              ~0, & \mbox{otherwise},
            \end{cases}
        \end{equation*}
      where $\xi \in \lp$ and $x\in X$. Then $V$ is a finite propagation operator, called a \emph{partial translation operator}.
\end{enumerate}
\end{ex}

The above two classes of operators indeed generate the set of all finite propagation operators as an algebra. More precisely, we have:
\begin{lem}[Lemma 2.4, \cite{vspakula2017metric}]\label{dec of BO}
Let $A$ be an operator on $\lp$ with propagation at most $r$, and $N=\sup_{x\in X} \sharp B(x,r)$. Then there exist multiplication operators $f_1,\ldots,f_N$ with $\|f_k\| \leq \sup_{x,y}\|A_{xy}\|$, and partial translation operators $V_1,\ldots,V_N$ of propagation at most $r$ such that:
$$A=\sum_{k=1}^N f_kV_k.$$
\end{lem}

Now we introduce our main objects: $\ell^p$-Roe-like algebras.
\begin{defn}
Let $(X,d)$ be a metric space with bounded geometry and $p \in \Dom$. Suppose $E$ is a Banach space and $B$ is a Banach subalgebra in $\Blp$ such that $C_b(X)BC_b(X)=B$ and $B$ is closed under block cutdowns. Define:
\begin{enumerate}[(i)]
  \item $\Roe$ to be the norm-closure of all the operators in $B$ with finite propagation, which is called the \emph{$\ell^p$-Roe-like algebra of $X$}; and
  \item $\K$ to be the norm-closure of $C_0(X)BC_0(X)$ in $\Blp$.
\end{enumerate}
\end{defn}

Recall that the $\ell^2$-Roe-like algebras were introduced by Tikuisis and the first author \cite{spakula2017relative} with an additional condition
\begin{equation}\label{commut condition 2.2}
  [C_0(X),B] \subseteq \K,
\end{equation}
which is used in the proof of their main result. Subsequently, Li, Wang and the second author \cite{LWZ2018} showed condition (\ref{commut condition 2.2}) is not necessary (for any $p\in[1,\infty)$).
An obvious advantage without assuming the commutant condition in the setup is to allow more examples, especially in the case of $p=1$ as shown in \cite[Example 1.10]{phillips2013crossed} and Example \ref{uniform Roe alg} below as well. Furthermore, similar to the proof of \cite[Lemma 2.10]{LWZ2018}, we have:
\begin{lem}\label{ideal}
For any $p \in \Dom$, $\K$ is an ideal in $\Roe$.
\end{lem}

Before immersing into examples of $\ell^p$-Roe-like algebras, let us recall the following notion related to matrix algebras.

\begin{defn}[\cite{phillips2013crossed}]
Let $(X,d)$ be a discrete metric space and $p \in \Dom$. Denote
$$\overline{M}^p_X:=\overline{C_c(X)\B(\ell^p(X))C_c(X)},$$
i.e., for any fixed point $x_0 \in X$
$$\overline{M}^p_X=\overline{\bigcup_{n\in \mathbb{N}}M^p_{B_n(x_0)}}$$
where $M^p_{B_n(x_0)}=\mathfrak{B}(\ell^p(B_n(x_0))) \subseteq \mathfrak{B}(\ell^p(X))$, which is the matrix algebra over the finite set $B_n(x_0)$.
\end{defn}

Phillips studied the relation between $\overline{M}^p_X$ and compact operators $\mathfrak{K}(\ell^p(X))$ in \cite{phillips2013crossed}. Following his arguments, we have:

\begin{lem}\label{K diff}
When $p \in \{0\} \cup (1,\infty)$, $\overline{M}^p_X=\mathfrak{K}(\ell^p(X))$; when $p \in \{1,\infty\}$, $\overline{M}^p_X \subsetneq \mathfrak{K}(\ell^p(X))$ in general.
\end{lem}

\begin{proof}
The first statement follows directly from \cite[Proposition 1.8]{phillips2013crossed}. For the second: when $p=\infty$ and $X$ is infinite, $\mathfrak{K}(\ell^\infty(X))$ is uncountable while $\overline{M}^\infty_X$ is countable, hence $\overline{M}^\infty_X \subsetneq \mathfrak{K}(\ell^\infty(X))$. When $p=1$, we follow \cite[Example 1.10]{phillips2013crossed}: take $X=\mathbb{N}$ and consider the operator $T: \ell^1(\mathbb{N}) \rightarrow \ell^1(\mathbb{N})$ defined by
$$T(\xi):=\big(\sum_{n\in \mathbb{N}} \xi(n)\big) \delta_0,$$
where $\xi \in \ell^1(\mathbb{N})$ and $\delta_0 \in \ell^1(\mathbb{N})$ is the function taking value $1$ at the original point, and $0$ elsewhere. Since $T$ has rank $1$, it belongs to $\mathfrak K (\ell^1(\mathbb{N}))$. However, it is clear that $T\notin \overline{M}^1_{\mathbb{N}}$.
\end{proof}

Now we provide several examples of $\ell^p$-Roe-like algebras, which include $\ell^p$-uniform Roe algebras, band-dominated operator algebras and $\ell^p$-Roe algebras; see \cite{LWZ2018,spakula2017relative} as well. Recall that $(X,d)$ is a metric space with bounded geometry, $E$ is a Banach space and $p \in \Dom$.

\begin{ex}[\textbf{$\ell^p$-Uniform Roe Algebra}]\label{uniform Roe alg}
Take $E=\mathbb{C}$ and $B=\mathfrak{B}(\ell^p(X))$, which is clearly closed under block cutdowns, and satisfies $C_b(X)BC_b(X)=B$. In this case, $\Roe$ is called \emph{the $\ell^p$-uniform Roe algebra} of $X$, which is defined in \cite{chung2018rigidity} and denoted by $B^p_u(X)$, and $\K=\overline{M}^p_X$ introduced above.

From Lemma \ref{K diff}, we know that when $p\in \{0\}\cup(1,\infty)$, $\overline{M}^p_X=\mathfrak{K}(\ell^p(X))$ and condition (\ref{commut condition 2.2}) holds. However, when $p \in \{1,\infty\}$, $\overline{M}^p_X \subsetneq \mathfrak{K}(\ell^p(X))$. Furthermore when $p=1$, the operator $T$ constructed in the proof of Lemma \ref{K diff} reveals that condition (\ref{commut condition 2.2}) does \emph{not} hold in general, since $[\delta_0,T] \notin \mathcal{K}(\mathbb{N}, \mathfrak{B}(\ell^1(\mathbb{N})))$.
\end{ex}

\begin{ex}[\textbf{Band-Dominated Operator Algebra}]\label{BD op}
Assume $(X,d)$ is strongly discrete, i.e., $\{d(x,y): x,y\in X\}$ is discrete. Take $B=\mathfrak{B}(\lp)$, which is clearly closed under block cutdowns and satisfies $C_b(X)BC_b(X)=B$. Elements in $B$ can be represented in the matrix form
$$b=(b_{x,y})_{x,y\in X} \in \mathfrak{B}(\lp), \quad \mbox{where} \quad b_{x,y} \in \mathfrak{B}(E).$$
Therefore $\Roe=\mathcal{A}^p_E(X)$, which is the band-dominated operator algebra (see \cite{vspakula2017metric,zhang2018}). And it is clear that $\K=\mathcal{K}^p_E(X)$ defined thereby, which is the set of all $\mathcal{P}$-compact operators on $\lp$.
\end{ex}

\begin{ex}[\textbf{$\ell^p$-Roe Algebra}]\label{$L^p$-Roe Algebra}
Recall that an operator $b$ in $\B(\ell^p(X; \ell^p(\mathbb{N}))) \cong \B(\ell^p(X \times \mathbb{N}))$ is \emph{locally compact} if for any $f \in C_0(X)$, $fb$ and $bf$ belong to $\mathfrak{K}(\ell^p(X \times \mathbb{N}))$.

Now take $E=\ell^p(\mathbb{N})$ and $B$ to be the set of all locally compact operators in $\B(\ell^p(X; \ell^p(\mathbb{N})))$, which is clearly closed under block cutdowns and satisfies $C_b(X)BC_b(X)=B$. The corresponding $\ell^p$-Roe-like algebra $\Roe$ is called \emph{the $\ell^p$-Roe algebra of $X$}, defined in \cite{chung2018rigidity} and denoted by $B^p(X)$. It is, by definition, the norm closure of all locally compact and finite propagation operators in $\mathfrak{B}(\ell^p(X; \ell^p(\mathbb{N})))$. When $p\in \{0\}\cup (1,\infty)$, $\K=\mathfrak{K}(\ell^p(X \times \mathbb{N}))$, which does not hold generally when $p \in \{1,\infty\}$.

Note that the $\ell^2$-Roe algebra coincides with the classical Roe algebra.
\end{ex}

\subsection{Property A and partition of unity}
Property A was first defined by Yu \cite{yu2000:CBC-for-CE}, and since then it has been shown to be equivalent to a plethora of other properties (including the metric sparsification property, and the operator norm localisation property \cite{brodzki2013uniform,chen2008metric,sako2014property}, cf.~also Section \ref{sec:p-in-1-infty}). We shall use the formulation in terms of existence of partitions of unity with small variation \cite[Theorem 1.2.4]{willett2009:notes-on-A}. For $p \in [1,\infty)$, set $q \in (1,\infty]$ to be the conjugate exponent of $p$, i.e., $1/p+1/q=1$.

First recall that for a metric space $X$, a cover $\mathcal{U}=\{U_i\}_{i \in I}$ is \emph{uniformly bounded} if $\sup_{i \in I} \diam(U_i)$ is finite; $\mathcal U$ has \emph{finite multiplicity} if there exists some $M$ such that for each $x\in X$, at most $M$ elements of $\mathcal U$ contain $x$. For a function $\phi$ on $X$, set its support to be the closure of $\{x: \phi(x) \neq 0\}$, denoted by $\supp(\phi)$.

\begin{defn}[Definition 6.1, \cite{vspakula2017metric}]
For $p \in [1,\infty)$, a \emph{metric $p$-partition of unity} on $X$ is a collection $\{\phi_i: X \to [0,1]\}_{i\in I}$ of functions on $X$ satisfying:
\begin{enumerate}
  \item The family $\{\supp(\phi_i)\}_{i \in I}$ is uniformly bounded and has finite multiplicity.
  \item For each $x \in X$, $\sum_{i \in I} \phi_i(x)^p=1$.
\end{enumerate}
Let $r,\varepsilon>0$. A metric $p$-partition of unity $\{\phi_i\}_{i\in I}$ has $(r,\varepsilon)$-\emph{variation} if for any $x,y \in X$ with $d(x,y) \leq r$, we have
$$\sum_{i \in I} |\phi_i(x)-\phi_i(y)|^p < \varepsilon^p.$$
The space $X$ has \emph{Property A} if for any $r,\varepsilon>0$, there exists a metric $p$-partition of unity with $(r,\epsilon)$-variation.
\end{defn}

Note that the above definition assumes $p\in[1,\infty)$. For the case that $p\in\{0,\infty\}$, we need a few more notions and a lemma from \cite{zhang2018}.

\begin{defn}[\cite{zhang2018}]
Let $\{\phi_i\}_{i \in I}$ be a metric $1$-partition of unity on $X$. A \emph{dual family} of $\{\phi_i\}_{i \in I}$ is defined to be a collection $\{\psi_i:X \to [0,1]\}_{i \in I}$ satisfying $\psi_i|_{\supp(\phi_i)}\equiv 1$ and there exists some $R>0$ such that $\supp(\psi_i) \subseteq \N_R(\supp(\phi_i))$ for any $i \in I$. A dual family $\{\psi_i\}_{i \in I}$ is \emph{$L$-Lipschitz}, if each $\psi_i$ is $L$-Lipschitz.
\end{defn}

Clearly for any $L>0$, $L$-Lipschitz dual family always exists. Furthermore, any dual family has finite multiplicity since $X$ has bounded geometry.

The following technical lemma is taken from \cite{zhang2018}, using partition of unity and the associated dual family to construct operators via blocks. To simplify the statement, we declare that a \emph{metric $p$-partition of unity} always means a metric $1$-partition of unity when $p\in \{0,1,\infty\}$.

\begin{lem}[\cite{zhang2018}]\label{commut converge}
For $p \in \Dom$, let $\{\phi_i\}_{i\in I}$ be a metric $p$-partition of unity on $X$, and let $\{\psi_i\}_{i \in I}$ be a dual family of $\{\phi_i\}_{i \in I}$ if $p\in \{0,1,\infty\}$. Given a collection of bounded linear operators $\{b_i\}_{i \in I}$ on $\lp$ such that $M:=\sup_{i}\|b_i\|$ is finite.
\begin{enumerate}
  \item When $p<\infty$, consider the following:
    \begin{enumerate}
       \item $\sum_{i \in I} \phi_i^{p/q}b_i\phi_i$ if $p \in (1,\infty)$;
       \item $\sum_{i \in I} \phi_i b_i \psi_i$ if $p=0$;
       \item $\sum_{i \in I} \psi_i b_i \phi_i$ if $p=1$.
    \end{enumerate}
    Each of them converges strongly to a finite propagation operator with norm at most $M$ on $\lp$.
  \item When $p=\infty$, consider $\sum_{i \in I} \phi_i b_i \psi_i$. It converges point-wise strongly to a finite propagation operator with norm at most $M$ on $\linf$.
\end{enumerate}
\end{lem}

\section{The main theorem}\label{sec:mainthm}

To state our main result, we need to recall two extra notions:

\begin{defn}[Definition 2.6, \cite{spakula2017relative}]
Let $(X,d)$ be a metric space. A bounded sequence $(f_n)_{n \in \mathbb N}$ in $C_b(X)$ is called \emph{very Lipschitz}, if for every $L>0$, there exists $n_0 \in \mathbb N$ such that $f_n$ is $L$-Lipschitz for any $n \geq n_0$. Let $\VL$ denote the set of all very Lipschitz bounded sequences from $C_b(X)$. Define
$$\VLinf:=\VL \big/ \big\{(f_n)_{n\in \mathbb N} \in \VL: \lim_{n \rightarrow \infty} \|f_n\|=0\big\}.$$
\end{defn}

\begin{defn}[\cite{roe2003lectures}]
Let $(X,d)$ be a proper metric space. A function $g \in C_b(X)$ is called a \emph{Higson function} (also called a \emph{slowly oscillating function}), if for every $R>0, \varepsilon>0$, there exists a compact set $A \subseteq X$ such that for any $x,y \in X\backslash A$ with $d(x,y) < R$, then $|g(x)-g(y)| < \varepsilon$. The set of all Higson functions on $X$ is denoted by $C_h(X)$.
\end{defn}

In the following, both $\VLinf$ and $B \subseteq \B(\lp)$ are regarded as Banach subalgebras of $\B(\lp)_\infty$. We consider the relative commutant:
$$B \cap \VLinf'.$$
It is clear that any operator in $\B(\lp)$ with finite propagation commutes with $\VLinf$. Hence by taking limits it follows that
\begin{equation}\label{iv to i}
\Roe \subseteq B \cap \VLinf'.
\end{equation}
The converse inclusion is also true provided the space has Property A, which is included in our main theorem as follows:

\begin{thm}\label{char for Roe alg Thm}
Let $(X,d)$ be a metric space with bounded geometry, and $p \in \Dom$. Suppose $E$ is a Banach space and $B \subseteq \Blp$ is a Banach subalgebra such that $C_b(X)BC_b(X)=B$ and closed under block cutdowns. For $b \in B$, consider the following conditions:
\begin{enumerate}
  \item[\emph{(i).}] $[b,f]=0$ for all $f \in \VLinf$;
  \item[\emph{(ii).}] $b$ is quasi-local;
  \item[\emph{(iii).}] $[b,g] \in \K$ for any $g \in C_h(X)$;
  \item[\emph{(iv).}] $b \in \Roe$.
\end{enumerate}
Then we have:
\begin{enumerate}
  \item when $p\in \{0,1,\infty\}$, (i) $\sim$ (iv) are equivalent;
  \item when $p\in (1,\infty)$, (i) $\sim$ (iii) are equivalent, and they are also equivalent to (iv) provided $X$ has Property A.
\end{enumerate}
\end{thm}

We state the result in this form to keep in line with \cite{LWZ2018,spakula2017relative}. Let us recall that \cite{spakula2017relative} proved the above result for $p=2$ and proper $X$ with straight finite decomposition complexity, under the extra assumption \eqref{commut condition 2.2} for Roe-like algebras. Subsequently, Li, Wang and the second author showed that \eqref{commut condition 2.2} is not needed and proved the result for arbitrary $p\in[1,\infty)$, under the same assumption on the space $X$. The Theorem above uses only a weaker assumption, namely Property A for $p\in[1,\infty)$, and no assumption on $X$ in case $p\in\{0,1,\infty\}$.


The equivalence that ``(i)$\Leftrightarrow$(ii)$\Leftrightarrow$(iii)" for $p \in \dom$, and the direction that ``(i) $\Leftrightarrow$ (ii) $\Rightarrow$ (iii)" for $p=\infty$ can be proved by similar arguments as those in \cite{spakula2017relative} for $p=2$ and in \cite{LWZ2018} for $p \in [1,\infty)$, hence omitted here. While special care needs to be taken for ``(iii) $\Rightarrow$ (i)" due to the fact that finitely supported vectors are no longer dense in $\linf$, which inevitably causes the problem that the norm of a general operator on $\linf$ might not be approximated by those of the restrictions on finite blocks. To overcome this problem, we provide the following auxiliary lemma. The proof is straightforward, hence omitted.

\begin{lem}
Let $p=\infty$ and $X,E,B$ be as in Theorem \ref{char for Roe alg Thm}, and $b \in \B(\linf)$ satisfy condition (iii) therein. Denote the set of all the finite subsets in $X$ by $\mathcal{F}$. Then for any $F_0 \in \mathcal{F}$, we have
$$\lim_{F\in \mathcal{F}}\|P_{F_0}b(\Id-P_F)\|=0 \andx \lim_{F\in \mathcal{F}}\|(\Id-P_F)bP_{F_0}\|=0.$$
In particular, we have
$$\|b\|=\sup_{F\in \mathcal{F}}\|P_FbP_F\|.$$
\end{lem}

Having observed the above lemma, now we can apply similar arguments as those in \cite{LWZ2018} to prove ``(iii) $\Rightarrow$ (i)" for $p=\infty$. And as we pointed out before, ``(iv) $\Rightarrow$ (i)" holds trivially for all $p \in \Dom$, hence we may only focus on ``(i) $\Rightarrow$ (iv)".

Before stepping into the detailed proof of ``(i) $\Rightarrow$ (iv)", we would like to substitute condition (i) with a more practical one. First let us recall the following notation: for any $p \in \Dom$ and $L,\varepsilon>0$, define
$$\Commut(L,\varepsilon)=\big\{a \in \Blp: \|[a,f]\| < \varepsilon, \mbox{~for~any~}L\mbox{-Lipschitz~} f\in C_b(X)_1\big\}.$$
As shown in \cite{spakula2017relative}, we have the following lemma:
\begin{lem}\label{commut replace}
Let $b \in \Blp$ and $\varepsilon>0$. Then $\|[b,f]\| < \varepsilon$ for every $f \in \VLinf_1$ if and only if there exists some $L>0$ such that $b \in \Commut(L,\varepsilon)$. Therefore, condition (i) in Theorem \ref{char for Roe alg Thm} can be replaced by the following: For any $\varepsilon>0$, there exists some $L>0$ such that $b \in \Commut(L,\varepsilon)$.
\end{lem}

\begin{proof}
Sufficiency is clear from the definition of $\VLinf$, so we only focus on necessity. Assume the contrary that for each $n$, there exists a $\frac{1}{n}$-Lipschitz $f_n$ in $C_b(X)_1$ such that $\|[b,f_n]\| \geq \varepsilon$. Then clearly $(f_n)_{n \in \mathbb N}$ belongs to $\VLinf_1$ and $\lim_{n\to \infty}\|[a,f_n]\| \geq \varepsilon$, which contradicts with the hypothesis.
\end{proof}

In the rest of the paper, we prove (i) implies (iv) in Theorem \ref{char for Roe alg Thm} (with the condition from Lemma \ref{commut replace} in place of (i)), and the proof is divided into two cases: $p\in\{0,1,\infty\}$ and $p \in (1,\infty)$. The cases share the general scheme (approximate a given quasi-local operator by a finite propagation one of the form from Lemma \ref{commut converge}). However, most of the work is concentrated on proving that these are close to the original operator in norm --- and for this the techniques in the two cases are rather different.

\section{Proof for $p\in \{0,1,\infty\}$}\label{sec:proof-01infty}

In this section, we prove part (1) of Theorem \ref{char for Roe alg Thm}, (i)$\implies$(iv). Notice that Property A is \emph{not} required here. We start with the following estimates on norms of commutant operators.

\begin{lem}\label{norm estimate}
For $p\in \{0,1,\infty\}$, let $b\in \Blp$ with $\|b\|=M$ and $b \in \Commut(L,\varepsilon)$ for some $L, \varepsilon>0$. Let $\{\phi_i\}_{i\in I}$ be a metric $1$-partition of unity on $X$, and $\{\psi_i\}_{i \in I}$ be an $L$-Lipschitz dual family of $\{\phi_i\}_{i \in I}$. Then:
\begin{enumerate}
  \item when $p=0$, the series $\sum_{i \in I} \phi_i [\psi_i,b]$ converges strongly to an operator on $\lz$ with norm at most $\varepsilon$;
  \item when $p=1$, the series $\sum_{i \in I} [\psi_i,b]\phi_i$ converges strongly to an operator on $\lo$ with norm at most $\varepsilon$;
  \item when $p=\infty$, the series $\sum_{i \in I} \phi_i [\psi_i,b]$ converges point-wise strongly to an operator on $\linf$ with norm at most $\varepsilon$.
\end{enumerate}
\end{lem}

\begin{proof}
(1) When $p=0$: first note that
$$\sum_{i \in I} \phi_i [\psi_i,b] = \sum_{i \in I} \phi_i\psi_ib - \sum_{i \in I} \phi_i b \psi_i = \sum_{i \in I} \phi_ib - \sum_{i \in I} \phi_i b \psi_i.$$
By Lemma \ref{commut converge}(1)(b) and the fact that $\sum_{i\in I}\phi_i$ converges strongly to the identity operator on $\lz$, we have that $\sum_{i \in I} \phi_i [\psi_i,b]$ converges strongly as well. Furthermore, for any vector $v \in \lz$ and $x\in X$, we have
\begin{equation}\label{eq:EQ13}
  \big\| \big(\sum_{i \in I} \phi_i [\psi_i,b]v\big)(x) \big\|_E = \big\| \sum_{i \in I} \phi_i(x) \big([\psi_i,b]v\big)(x) \big\|_E
  \leq  \sum_{i \in I}\phi_i(x)\cdot\big\|[\psi_i,b] v\big\|_\infty.
\end{equation}
Note that each $\psi_i$ is $L$-Lipschitz and $b \in \Commut(L,\varepsilon)$, hence $\|[\psi_i,b]\|<\varepsilon$ for $i \in I$. Therefore, we have
\begin{equation}\label{eq:EQ14}
\big\| \sum_{i \in I} \phi_i [\psi_i,b]v\big\|_\infty \leq \sum_{i \in I}\phi_i(x)\varepsilon \|v\|_\infty = \varepsilon\|v\|_\infty.
\end{equation}

(2) When $p=1$: notice that
$$\sum_{i \in I} [\psi_i,b]\phi_i = \sum_{i \in I} \psi_i b \phi_i - \sum_{i \in I} b \psi_i\phi_i = \sum_{i \in I} \psi_i b \phi_i - \sum_{i \in I} b \phi_i.$$
By Lemma \ref{commut converge}(1)(c) and the fact that $\sum_{i\in I}\phi_i$ converges strongly to the identity operator on $\lo$, we have that $\sum_{i \in I} [\psi_i,b]\phi_i$ converges strongly as well. Furthermore, for any vector $v \in \lo$ and unit vector $w$ in $\ell^\infty_{E^*}(X) \cong \lo^*$, we have
\begin{equation*}
  \big| \big\langle \sum_{i \in I} [\psi_i,b] \phi_iv, w \big\rangle \big| \leq  \sum_{i\in I} \big|\big\langle \phi_iv, [\psi_i,b]^* w\big\rangle\big| \leq \sum_{i\in I} \|\phi_iv\|_1 \cdot \|[\psi_i,b]^*\|.
\end{equation*}
Note that each $\psi_i$ is $L$-Lipschitz and $b \in \Commut(L,\varepsilon)$, hence we have
$\|[\psi_i,b]^*\|=\|[\psi_i,b]\|<\varepsilon$. On the other hand, notice that
$$\sum_{i\in I} \|\phi_iv\|_1 = \sum_{i \in I} \sum_{x\in X} \phi_i(x)\|v(x)\|_E = \sum_{x\in X}\big( \sum_{i \in I}\phi_i(x) \big) \|v(x)\|_E = \sum_{x\in X} \|v(x)\|_E = \|v\|_1.$$
Therefore,
\begin{equation*}
\big\| \sum_{i \in I} [\psi_i,b] \phi_iv \big\|_\infty \leq \sum_{i\in I} \|\phi_iv\|_1 \cdot \varepsilon = \varepsilon \|v\|_1.
\end{equation*}

(3) When $p=\infty$: again we have
$$\sum_{i \in I} \phi_i [\psi_i,b] = \sum_{i \in I} \phi_i\psi_ib - \sum_{i \in I} \phi_i b \psi_i = \sum_{i \in I} \phi_ib - \sum_{i \in I} \phi_i b \psi_i.$$
By Lemma \ref{commut converge}(2) and the fact that $\sum_{i\in I}\phi_i b$ converges point-wise strongly to $b$ in $\B(\linf)$, we know that $\sum_{i \in I} \phi_i [\psi_i,b]$ converges point-wise strongly as well. Furthermore, (\ref{eq:EQ13}) and (\ref{eq:EQ14}) still hold, so we finish the proof.
\end{proof}

The following lemma is well-known to coarse geometers (see e.g.~\cite[Lemma (7.1)]{roe1988:index-thm-on-open-mfds}). While for completeness, we provide a proof as well.
\begin{lem}\label{bdd geo}
Let $(X,d)$ be a metric space with bounded geometry, and $\{U_i\}_{i \in I}$ be a uniformly bounded family of subsets in $X$ with finite multiplicity. Then there exists a natural number $N$ and a decomposition $I=\bigsqcup_{k=1}^N I_k$, such that elements in each subfamily $\{U_i\}_{i \in I_k}$ are mutually disjoint for $k=1,\ldots, N$.
\end{lem}

\begin{proof}
We construct a graph $G$ whose vertex set $V$ is $\{U_i\}_{i \in I}$, and two vertices $U_i, U_j$ are connected by an edge if and only if $U_i \cap U_j \neq \emptyset$ and $i \neq j$. Since $X$ has bounded geometry and the family $\{U_i\}_{i \in I}$ is uniformly bounded and has finite multiplicity, the graph $G$ has finite valency $N$ for some $N \in \mathbb N$. Clearly, it suffices to divide $V$ into finitely many subsets such that any two vertices from the same one are not connected by an edge.

To do so, first claim that there exists a subset $V' \subseteq V$ such that vertices in $V'$ are not connected with each other, and after deleting $V'$ as well as any edge with at least one endpoint in $V'$, the remained graph $G'$ has valency at most $N-1$. In fact, $V'$ can be chosen as follows: take a vertex $v$ with valency $N$ and put it into $V'$, then delete all the edges connecting $v$ as well as $v$ itself from $G$ and consider the remained graph. It is clear that any vertex with valency $N$ in this new graph are not connected with $v$ in $G$. Repeat this procedure and by Zorn's lemma, we end up with a required $V'$. So the claim is proved.

Now we do the same procedure to the new graph $G'$ with valency at most $N-1$, and step by step, we end up with a decomposition of $V$ into at most $N$ subsets satisfying the requirement. Hence we finish the proof.
\end{proof}

\begin{proof}[Proof of Theorem \ref{char for Roe alg Thm} (1), ``(i) $\Rightarrow$ (iv)"]
By Lemma \ref{commut replace}, we may suppose that $b \in B \subseteq \B(\lp)$ satisfies the condition: for any $\varepsilon>0$, there exists some $L>0$, such that $b \in \Commut(L,\varepsilon)$. Now fix an $\varepsilon>0$, and let $L$ be as in the condition. Take an arbitrary family of metric $1$-partition of unity $\{\phi_i\}_{i \in I}$. Such a family always exists: for example, take an arbitrary disjoint bounded cover $\{U_i\}_{i \in I}$ of $X$, and take $\phi_i$ to be the characteristic function of $U_i$. We also take an $L$-Lipschitz dual family $\{\psi_i\}_{i \in I}$ of $\{\phi_i\}_{i \in I}$. Now define:
\begin{equation*}
  b_\varepsilon=
    \begin{cases}
      ~\sum_{i \in I} \phi_i b \psi_i, & \mbox{~when~} p=0 \mbox{~or~} \infty, \\
      ~\sum_{i \in I} \psi_i b \phi_i, & \mbox{~when~} p=1.
    \end{cases}
\end{equation*}
By Lemma \ref{commut converge}, each of the series above converge strongly or point-wise strongly, and $b_\varepsilon$ is a finite propagation operator in $\Blp$.

Applying Lemma \ref{bdd geo} to the cover $\{\supp(\psi_i)\}_{i \in I}$ provides a natural number $N$ and a decomposition $I=\bigsqcup_{k=1}^N I_k$ satisfying the requirement thereby. Hence for each $k=1,\ldots,N$, $\sum_{i\in I_k}\psi_i b\psi_i \in B$ since $B$ is closed under block cutdown. Consider the function $\phi^{(k)}:=\sum_{i \in I_k}\phi_i$ which belongs to $C_b(X)$, and we have that
$$\sum_{i\in I_k}\phi_i b\psi_i = \phi^{(k)} \sum_{i\in I_k}\psi_i b\psi_i$$
belongs to $B$ since $C_b(X)BC_b(X)=B$. Therefore,
$$\sum_{i \in I} \phi_i b \psi_i = \sum_{k=1}^N \big(\sum_{i\in I_k}\phi_i b\psi_i\big)$$
belongs to $B$ as well. Similarly, we have $\sum_{i \in I} \psi_i b \phi_i \in B$. Therefore, we obtain that $b_\varepsilon \in B$ for $p \in \{0,1,\infty\}$.

Finally, notice that
\begin{equation*}
  b-b_\varepsilon=
    \begin{cases}
      ~\sum_{i \in I} \phi_i\psi_i b - \sum_{i \in I} \phi_i b \psi_i = \sum_{i \in I} \phi_i[\psi_i,b], & \mbox{~when~} p=0 \mbox{~or~} \infty, \\
      ~\sum_{i \in I}  b \psi_i\phi_i- \sum_{i \in I} \psi_i b \phi_i = \sum_{i \in I} [b,\psi_i]\phi_i, & \mbox{~when~} p=1.
    \end{cases}
\end{equation*}
Hence according to Lemma \ref{norm estimate}, $\|b-b_\varepsilon\| \leq \varepsilon$, which implies that $b \in \Roe$.
\end{proof}

\section{Proof for $p \in (1,\infty)$}\label{sec:p-in-1-infty}

This section contains the proof of part (2) in Theorem \ref{char for Roe alg Thm}, (i)$\implies$(iv). Property A will appear here in two ways: To ensure the existence of metric partitions of unity with small variation, and as the metric sparsification property introduced in \cite{chen2008metric}. For the equivalence of the two formulations, see \cite{sako2014property}. Note that we only use the fact that Property A implies the metric sparsification property, which was done in an elementary way in \cite{brodzki2013uniform}.

\begin{defn}
Let $(X,d)$ be a metric space. Then $X$ has the \emph{metric sparsification property}, if for any $c \in (0,1]$, there exists a non-decreasing function $f: \mathbb N \to \mathbb N$ such that for any $m \in \mathbb N$ and any finite positive Borel measure $\mu$ on $X$, there is a Borel subset $\Omega=\bigsqcup_{i \in I}\Omega_i$ of $X$ such that
\begin{itemize}
  \item $d(\Omega_i,\Omega_j) \geq m$, whenever $i \neq j$;
  \item $\diam(\Omega_i) \leq f(m)$ for every $i \in I$;
  \item $\mu(\Omega) \geq c \mu(X)$.
\end{itemize}
\end{defn}

The following lemma is a key ingredient, which can be regarded as a \emph{commutant} version of the operator norm localisation property (see Definition \ref{ONL def} below).
\begin{lem}\label{ONL}
Assume that $X$ has the metric sparsification property, and $p \in (1,\infty)$. Then for any $\varepsilon>0$, $L>0$ and $M>0$, there exists $s>0$ such that for any operator $b \in \Blp$ with $\|b\|\leq M$ and $b \in \Commut(L,\varepsilon)$, there exists a unit vector $v \in \lp$ with $\diam(\supp v) \leq s$, and satisfying:
$$\|bv\| \geq \|b\|-6\varepsilon.$$
\end{lem}

The proof is inspired by \cite[Proposition 4.1]{chen2008metric} together with the following technical result introduced first \cite{spakula2017relative} in the case of $p=2$ and later generalised \cite{LWZ2018} to $p \in [1,\infty)$, which plays an key role here. Briefly speaking, it provides an approach to decompose operators into blocks.

\begin{prop}[\cite{LWZ2018,spakula2017relative}]\label{estimate for commut}
For $p \in [1,\infty)$, let $a$ be an operator on $\lp$ and $a \in \Commut(L,\varepsilon)$ for some $L,\varepsilon>0$. Let $(e_i)_{i\in I}$ be a family of positive contractions in $C_b(X)$ with $2/L$-disjoint supports, and define $e:=\sum_{i\in I} e_i$. Then we have
$$\big\|eae-\sum_{i\in I}e_iae_i\big\| \leq \varepsilon.$$
\end{prop}

\begin{proof}[Proof of Lemma \ref{ONL}]
Let us fix $\varepsilon,L,M>0$, and the proof is divided into two steps.

\emph{Step I.} Suppose a vector $v \in \lp$ has the form $v=\sum_{i \in I}v_i$ with
$$d(\supp(v_i),\supp(v_j)) > \frac{4}{L}.$$
The aim of this step is to prove the inequality \eqref{EQ5'} below.

For each $i\in I$, set $Y_i:=\supp(v_i)$ and take a positive $L$-Lipschitz contraction $f_i\in C_b(X)$ with $f_i|_{Y_i}=1$ and $\supp(f_i) \subseteq \N_{1/L}(Y_i)$. Hence, the family $\{\supp(f_i):i\in I\}$ is $2/L$-disjoint. Taking $f=\sum_{i\in I}f_i \in C_b(X)_1$, then $f$ is $L$-Lipschitz and $fv=v$. Since $b \in \Commut(L,\varepsilon)$ and $1-f$ is $L$-Lipschitz, we have
\begin{equation}\label{EQ1}
\|bv\| \leq \|fbv\| + \|[1-f,b]v\| + \|b(1-f)v\| \leq \|fbv\| + \varepsilon\|v\|.
\end{equation}
Applying Proposition \ref{estimate for commut} to $\{f_i\}_{i\in I}$ and the operator $b$, we have $\|fbf - \sum_{i \in I}f_i b f_i\| \leq \varepsilon$, which implies that
\begin{equation}\label{EQ2}
\|fbv\|=\|fbfv\| \leq \big\|\sum_{i \in I}f_i b f_iv\big\| + \varepsilon\|v\| = \big\|\sum_{i \in I}f_i bv_i\big\| + \varepsilon\|v\|. \end{equation}
Using the triangle inequality (in the space $\ell^p(I,\lp)$), and the fact that $\{f_i\}_{i \in I}$ have mutually disjoint supports, we have
\begin{eqnarray}\label{EQ3}
\big\|\sum_{i \in I}f_i b v_i\big\|& = & \big(\sum_{i\in I}\|f_ibv_i\|^p\big)^{\frac1p} = \big(\sum_{i\in I}\|bv_i-(1-f_i)bv_i\|^p\big)^{\frac1p} \nonumber \\
& \leq & \big(\sum_{i\in I} \|bv_i\|^p\big)^{\frac{1}{p}} + \big(\sum_{i\in I} \|(1-f_i) bv_i\|^p\big)^{\frac{1}{p}} \nonumber \\
& \leq & \big(\sum_{i\in I} \|bv_i\|^p\big)^{\frac{1}{p}} + \big(\sum_{i\in I} \varepsilon^p \|v_i\|^p\big)^{\frac{1}{p}} \nonumber \\
& = & \big(\sum_{i\in I} \|bv_i\|^p\big)^{\frac{1}{p}} + \varepsilon \|v\|,
\end{eqnarray}
where $\|[1-f_i,b]\|< \varepsilon$ is used in the third line. Combining (\ref{EQ1})$\sim$(\ref{EQ3}), we have
\begin{equation}\label{EQ4}
(\|bv\|- 3\varepsilon \|v\|)^p \leq \sum_{i\in I} \|bv_i\|^p.
\end{equation}

Now we claim:
\begin{equation}\label{EQ5'}
\frac{\|bv\|}{\|v\|} \leq \sup_{i \in I} \frac{\|bv_i\|}{\|v_i\|} + 3\varepsilon.
\end{equation}
If it were not true, then for any $i \in I$:
$$\frac{(\|bv\|-3\varepsilon\|v\|)^p}{\|v\|^p} > \frac{\|bv_i\|^p}{\|v_i\|^p}.$$
Combining with (\ref{EQ4}), we have
$$(\|bv\|- 3\varepsilon \|v\|)^p  \leq  \sum_{i\in I} \|bv_i\|^p < \sum_{i\in I} \frac{(\|bv\|-3\varepsilon\|v\|)^p\|v_i\|^p}{\|v\|^p} = (\|bv\|-3\varepsilon\|v\|)^p,$$
which is a contradiction.

\emph{Step II.} Since $X$ has the metric sparsification property, for any $c \in (0,1)$, there exists a function $f: \mathbb N \to \mathbb N$ such that for any finite positive Borel measure $\mu$ on $X$, there exists a decomposition $\Omega=\bigsqcup_{i\in I} \Omega_i \subseteq X$, satisfying: $d(\Omega_i, \Omega_j) > 4/L$ for any $i \neq j$, $\diam(\Omega_i) \leq f(4/L)$, and $\mu(\Omega) \geq c\mu(X)$. For any $w \in \lp \setminus \{0\}$, we define a measure $\mu$ on $X$ by $\mu(\{x\})=\|w(x)\|^p$, and let $\Omega=\bigsqcup_{i\in I} \Omega_i$ be the associated decomposition satisfying the above conditions. Then we have
$$\|bw-bP_\Omega w\|^p \leq \|b\|^p \cdot \|w-P_\Omega w\|^p = \|b\|^p \mu(X \setminus \Omega) \leq M^p(1-c)\|w\|^p,$$
which implies that
\begin{equation}\label{EQ6}
\|bP_\Omega w\| \geq \|bw\| - M(1-c)^{\frac{1}{p}}\|w\|.
\end{equation}
Note that $P_\Omega w = \sum_{i\in I}P_{\Omega_i}w$ has the form in Step I. Hence by (\ref{EQ5'}) and (\ref{EQ6}), we have:
\begin{equation*}
\sup_{i \in I} \frac{\|bP_{\Omega_i}w\|}{\|P_{\Omega_i}w\|} +3\varepsilon \geq \frac{\|bP_\Omega w\|}{\|P_\Omega w\|} \geq \frac{\|bw\| - M(1-c)^{\frac{1}{p}}\|w\|}{\|w\|} = \frac{\|bw\|}{\|w\|} - M(1-c)^{\frac{1}{p}}.
\end{equation*}
Now take a vector $w \in \lp\setminus \{0\}$ such that $\frac{\|bw\|}{\|w\|} \geq \|b\|-\varepsilon$, then we have
\begin{equation}\label{EQ7}
\sup_{i \in I} \frac{\|bP_{\Omega_i}w\|}{\|P_{\Omega_i}w\|} \geq \|b\| - 4\varepsilon  - M(1-c)^{\frac{1}{p}}.
\end{equation}
Finally, notice that $\lim_{c \to 1}M(1-c)^{\frac{1}{p}}=0$, we may take $c \in (0,1)$ at the beginning of Step II satisfying $M(1-c)^{\frac{1}{p}} < \varepsilon$, and let $f$ be the associated function thereby. Then there exists some $i \in I$, such that
$$\frac{\|bP_{\Omega_i}w\|}{\|P_{\Omega_i}w\|} \geq \|b\| - 6\varepsilon,$$
with $\diam(\supp(P_{\Omega_i}w)) \leq f(4/L)$. Setting $s:=f(4/L)$ and $v= \frac{P_{\Omega_i}w}{\|P_{\Omega_i}w\|}$, we finish the proof.
\end{proof}

To prove the main theorem, we would like to approximate an operator $b \in \Blp$ by block diagonal operators in the form of $\sum_{i \in I}\phi_i^{p/q}b\phi_i$ coming from Lemma \ref{commut converge}(1)(a), where $\{\phi_i\}_{i \in I}$ is a metric $p$-partition of unity and $q$ is the conjugate exponent of $p$. To estimate their difference, we calculate as follows:
\begin{equation}\label{EQ8}
\sum_{i \in I} \phi_i^{p/q}b\phi_i - b = \sum_{i \in I} \phi_i^{p/q}b\phi_i - \sum_{i \in I}\phi_i^{p/q+1}b =\sum_{i \in I} \phi_i^{p/q}[b,\phi_i],
\end{equation}
which converges strongly as well. Furthermore, we have the following uniform control:

\begin{lem}\label{unif commut}
Let $b \in \Commut(L,\varepsilon)$ for some $L,\varepsilon>0$ and $p \in (1,\infty)$. Then for any metric $p$-partition of unity $\{\phi_i\}_{i \in I}$, the operator $\sum_{i \in I} \phi_i^{p/q}[b,\phi_i]$ belongs to $\Commut(L,2\varepsilon)$.
\end{lem}

\begin{proof}
For any $L$-Lipschitz contraction $f\in C_b(X)$, (\ref{EQ8}) implies that
$$\big[\sum_{i \in I} \phi_i^{p/q}[b,\phi_i] , f\big] = \big[\sum_{i \in I} \phi_i^{p/q}b\phi_i , f\big] - [b,f] = \sum_{i \in I} \phi_i^{p/q}[b,f]\phi_i - [b,f].$$
Since $b \in \Commut(L,\varepsilon)$, then $\|[b,f]\|<\varepsilon$. Hence by Lemma \ref{commut converge}(1)(a), we have
$$\big\| \sum_{i \in I} \phi_i^{p/q}[b,f]\phi_i \big\| < \varepsilon$$
Therefore,
$$\big\| \big[\sum_{i \in I} \phi_i^{p/q}[b,\phi_i] , f\big] \big\| < \big\| \sum_{i \in I} \phi_i^{p/q}[b,f]\phi_i \big\| + \|[b,f]\| < \varepsilon + \varepsilon = 2\varepsilon,$$
which finishes the proof.
\end{proof}

\begin{proof}[Proof of Theorem \ref{char for Roe alg Thm} (2), (i)$\implies$(iv)]
Fixing an operator $b \in B$ satisfying condition (i) and $\varepsilon>0$, we aim to $\varepsilon$-approximate $b$ by an operator in $B$ with finite propagation.

Setting $M=\|b\|$, there exists $L>0$ such that $b \in \Commut(L,\frac{\varepsilon}{\max\{4M,24\}})$ by Lemma \ref{commut replace}. Applying Lemma \ref{ONL} to $\varepsilon/12, L, 2M$, we obtain an $s>0$ such that for any operator $a \in \Blp$ with $\|a\|\leq 2M$ and $a \in \Commut(L,\varepsilon/12)$, there exists a unit vector $v \in \lp$ with $\diam(\supp v) \leq s$, and satisfying $\|av\| \geq \|a\|-\varepsilon/2$. Since $X$ has bounded geometry, the number $K:=\sup_{z\in X} \sharp B(z,s+\frac{1}{L})$ is finite.

Now we may take a metric $p$-partition of unity $\{\phi_i\}_{i \in I}$ with $(s+\frac{2}{L},\frac{\varepsilon}{4MK})$-variation since $X$ has Property A, and consider the series
$$b':=\sum_{i \in I} \phi_i^{p/q}b\phi_i.$$
By Lemma \ref{commut converge}(1)(a), it converges strongly in $\Blp$ with $\|b'\| \leq \|b\| \leq M$ and $b'$ has finite propagation. Hence the operator $a:=b'-b$ has norm at most $2M$, and $a=\sum_{i \in I}\phi_i^{p/q}[b,\phi_i]$ by (\ref{EQ8}). Furthermore by Lemma \ref{unif commut}, $a \in \Commut(L,\varepsilon/12)$. Hence there exists a unit vector $v \in \lp$ with $\diam(\supp v) \leq s$ and satisfying
\begin{equation}\label{eq:EQ12}
\|av\| \geq \|a\|-\varepsilon/2.
\end{equation}
Take $F:=\supp(v)$, we have:
\begin{equation}\label{eq:EQ11}
\|av\| \leq \|\chi_{\N_{1/L}(F)}av\| + \|\chi_{\N_{1/L}(F)^c}a\chi_Fv\| \leq \|\chi_{\N_{1/L}(F)}a\chi_Fv\| + \varepsilon/12.
\end{equation}
Now for any $x\in \N_{1/L}(F)$, we calculate:
\begin{eqnarray}\label{eq:EQ9}
  \| (av)(x) \| &=& \big\| \sum_{y \in F}a_{xy}v(y) \big\| = \big\| \sum_{y\in F} \sum_{i \in I} \phi_i^{p/q}(x)([b,\phi_i])_{xy}v(y) \big\| \nonumber\\
   &=&  \big\| \sum_{y\in F} \sum_{i \in I} \phi_i^{p/q}(x)b_{xy}(\phi_i(y)-\phi_i(x))v(y) \big\|.
\end{eqnarray}
For each $y \in F$, we calculate (using the H\"older inequality on the second line):
\begin{eqnarray}\label{eq:EQ10}
  &&\big\| \sum_{i \in I} \phi_i^{p/q}(x)b_{xy}(\phi_i(y)-\phi_i(x))\big\| \leq M\sum_{i \in I}\phi_i^{p/q}(x) |\phi_i(y)-\phi_i(x)| \nonumber \\
   &\leq & M \big( \sum_{i \in I} \phi_i^{p}(x)\big)^{1/q} \cdot \big( \sum_{i \in I} |\phi_i(y)-\phi_i(x)|^p\big)^{1/p} \leq M \cdot \frac{\varepsilon}{4MK} = \frac{\varepsilon}{4K}
\end{eqnarray}
since $d(x,y) \leq s+\frac{2}{L}$, and $\{\phi_i\}_{i \in I}$ has $(s+\frac{2}{L},\frac{\varepsilon}{4MK})$-variation. Combining \eqref{eq:EQ9} with \eqref{eq:EQ10}, we have (again using the H\"older inequality on the second line)
\begin{eqnarray*}
\| (a\chi_Fv)(x) \| & \leq & \sum_{y\in F} \big\| \sum_{i \in I} \phi_i^{p/q}(x)b_{xy}(\phi_i(x)-\phi_i(y)) \big\|\cdot\|v(y)\| \\
&\leq & \frac{\varepsilon}{4K}\sum_{y\in F}\|v(y)\| \leq \frac{\varepsilon}{4K} (\sharp F)^{1/q} \cdot \big( \sum_{y \in F} \|v(y)\|^p\big)^{1/p}\\
& \leq & \frac{\varepsilon K^{1/q}}{4K} \|v\| = \frac{\varepsilon K^{1/q}}{4K}.
\end{eqnarray*}
Combining with \eqref{eq:EQ11}, we have
$$\|av\|\leq \|\chi_{\N_{1/L}(F)}a\chi_Fv\| + \frac{\varepsilon}{12} =\big( \sum_{x\in \N_{1/L}(F)}\|(a\chi_Fv)(x)\|^p \big)^{1/p} + \frac{\varepsilon}{12} \leq \frac{\varepsilon K^{1/q}}{4K} \cdot K^{1/p} + \frac{\varepsilon}{12} < \frac{\varepsilon}{2}. $$
Therefore, from \eqref{eq:EQ12}, we obtain that
$$\|a\| \leq \|av\| + \frac{\varepsilon}{2} < \frac{\varepsilon}{2} + \frac{\varepsilon}{2}=\varepsilon.$$

Finally, we prove that $b' \in B$. Applying Lemma \ref{bdd geo} to the cover $\{\supp(\phi_i)\}_{i \in I}$ provides a natural number $N$ and a decomposition $I=\bigsqcup_{k=1}^N I_k$ satisfying the requirement thereby. Hence for each $k=1,\ldots,N$, $\sum_{i\in I_k}\phi_i b\phi_i$ and $\sum_{i\in I_k}\phi_i^{p/q} b\phi_i^{p/q}$ belong to $B$ since it is closed under block cutdown. Therefore if $p=2$: $\sum_{i \in I_k} \phi_i^{p/q}b\phi_i = \sum_{i\in I_k}\phi_i b\phi_i \in B$. If $p\in (1,2)$: consider the function $\phi^{(k)}:=\sum_{i \in I_k}\phi_i^{1-p/q} \in C_b(X)$, and we have that $\sum_{i\in I_k}\phi_i^{p/q}b\phi_i = \big(\sum_{i\in I_k}\phi_i^{p/q} b\phi_i^{p/q}\big) \phi^{(k)}$ belongs to $B$ since $C_b(X)BC_b(X)=B$. If $p\in (2,\infty)$: consider the function $\phi^{(k)}:=\sum_{i \in I_k}\phi_i^{p/q-1} \in C_b(X)$, and we have that $\sum_{i\in I_k}\phi_i^{p/q}b\phi_i = \phi^{(k)} \big(\sum_{i\in I_k}\phi_i b\phi_i\big)$ belongs to $B$ as well. Therefore,
$$b'=\sum_{i \in I} \phi_i^{p/q}b\phi_i = \sum_{k=1}^N \big(\sum_{i\in I_k}\phi_i^{p/q}b\phi_i\big)$$
belongs to $B$ in all cases. So we finish the whole proof.
\end{proof}

\section{Applications and Questions}\label{sec:apps-and-qs}

\subsection{A characterisation of Property A}\label{subsec:charact-of-A}

The aim of this subsection is to prove that the conclusion of Lemma \ref{ONL} for $p=2$ and $E=\mathbb{C}$ is actually equivalent to the operator norm localisation property, introduced by Chen, Tessera, Wang and Yu \cite{chen2008metric}. Therefore, it provides an alternative approach to characterise Property A by a result of Sako \cite{sako2014property}.

We recall a simplified version of the original definition for the operator norm localisation property; their equivalence is proved in \cite[Proposition 3.1]{sako2014property}.

\begin{defn}\label{ONL def}
Let $X$ be a metric space with bounded geometry. We say $X$ has the \emph{operator norm localisation} property, if for any $c\in (0,1)$ and $R>0$, there exists $S>0$ such that for any $a \in \B(\ell^2(X))$ with propagation at most $R$ and norm $1$, there exists a unit vector $v \in \ell^2(X)$ satisfying $\diam(\supp(v)) \leq S$ and $\|av\| \geq c$.
\end{defn}

Now we need the following auxiliary lemma.
\begin{lem}\label{bd to commut}
For any $\varepsilon>0$ and $R>0$, there exists a constant $L>0$, such that for any operator $a$ in $\B(\ell^2(X))$ with propagation at most $R$ and norm $1$, we have $a \in \Commut(L,\varepsilon)$.
\end{lem}

\begin{proof}
For $a \in \B(\ell^2(X))$ with propagation at most $R$, take $N:=\sup_{x\in X} \sharp B(x,R)$. By Lemma \ref{dec of BO}, there exist multiplication operators $f_1,\ldots,f_N$ with $\|f_k\| \leq 1$, and partial translation operators $V_1,\ldots,V_N$ of propagation at most $R$ such that $a=\sum_{k=1}^N f_kV_k$. Denote by $t_k:D_k\to R_k$ the partial translation corresponding to $V_k$ for each $k$. Taking $L=\frac{\varepsilon}{NR}$ and for any $L$-Lipschitz contraction $f\in C_b(X)$, we have
$$[a,f]= \sum_{k=1}^N f_k[V_k,f].$$
For each $k=1,\ldots, N$, $v \in \ell^2(X)$ and $x\in X$, we calculate:
\begin{equation}\label{EQ5}
\big( [f,V_k]v \big)(x)=
\begin{cases}
  ~\big(f(x)-f(t_k^{-1}(x))\big)v(t_k^{-1}(x)), & x\in R_k; \\
  ~0, & \mbox{otherwise}.
\end{cases}
\end{equation}
Hence,
$$\|[f,V_k]v\| = \big(\sum_{x\in R_k} \|(f(x)-f(t_k^{-1}(x))v(t_k^{-1}(x))\|^2\big)^{1/2} \leq \frac{\varepsilon}{N} \|v\|,$$
since $f$ is $\frac{\varepsilon}{NR}$-Lipschitz and $d(x,t_k^{-1}(x)) \leq R$. Therefore,
$$\|[a,f]\| \leq \sum_{k=1}^N \|[V_k,f]\| \leq N \cdot \frac{\varepsilon}{N} = \varepsilon,$$
which implies that $a \in \Commut(L,\varepsilon)$. We finish the proof.
\end{proof}

Now we present and prove the following characterisation for Property A.
\begin{prop}\label{prop:charact-of-A}
Let $(X,d)$ be a metric space with bounded geometry. The following are equivalent:
\begin{enumerate}[(1)]
  \item $X$ has Property A;
  \item $X$ has the metric sparsification property;
  \item\label{item:char-of-A-eps} For any $\varepsilon>0$, $L>0$ and $M>0$, there exists $s>0$, such that for any operator $b \in \B(\ell^2(X))$ with $\|b\|\leq M$ and $b \in \Commut(L,\frac{\varepsilon}{6})$, there exists a unit vector $v \in \ell^2(X)$ with $\diam(\supp v) \leq s$, and satisfying $\|bv\| \geq \|b\|-\varepsilon$.
  \item\label{item:char-of-A-c} For any $c\in (0,1)$, $L>0$, there exists $s>0$, such that for any operator $b \in \B(\ell^2(X))$ with $\|b\|=1$ and $b \in \Commut(L, \frac{1-c}{6})$, there exists a unit vector $v \in \ell^2(X)$ with $\diam(\supp v) \leq s$, and satisfying $\|bv\| \geq c$.
  \item $X$ has the operator norm localisation property.
\end{enumerate}
\end{prop}

\begin{proof}
``(1) $\Rightarrow$ (2)" follows from \cite[Proposition 3.2 and 3.8]{brodzki2013uniform}. ``(2) $\Rightarrow$ (3)" follows from Lemma \ref{ONL} in the case of $p=2$ and $E=\mathbb C$. Let us start with ``(3) $\Rightarrow$ (4)". Fix $c \in (0,1)$ and $L>0$. Taking $\varepsilon=1-c>0$ and $M=1$, there exists $s>0$ satisfying the property in condition (3). Hence for any $b\in \Commut(L, \frac{1-c}{6})=\Commut(L,\frac{\varepsilon}{6})$ with norm $1$, there exists a unit vector $v \in \ell^2(X)$ with $\diam(\supp v) \leq s$ and satisfying $\|bv\| \geq 1-\varepsilon=1-(1-c)=c$.

Now we prove ``(4) $\Rightarrow$ (5)". Given $c \in (0,1)$ and $R>0$, applying Lemma \ref{bd to commut} to $\varepsilon=\frac{1-c}{6}$ and $R$ produces a constant $L>0$ such that for any $b\in \B(\ell^2(X))$ with propagation at most $R$ and norm $1$, $b \in \Commut(L,\frac{1-c}{6})$. Applying condition (4) to the above $c$ and $L$, we obtain a constant $s$ such that for the above $b$, there exists a unit vector $v \in \ell^2(X)$ with $\diam(\supp v) \leq s$ and satisfying $\|bv\| \geq c$. Hence $X$ has operator norm localisation property.

Finally, ``(5) $\Rightarrow$ (1)" is proved in \cite[Theorem 4.1]{sako2014property}.
\end{proof}

\subsection{Inverses in $\ell^p$-Roe-like algebras}

Recall that unital $C^*$-algebras are always closed under taking inverses. More precisely, for any two unital $C^*$-algebras $A \subseteq B$, if an element $a\in A$ is invertible in $B$, then $a^{-1} \in A$. This does not hold generally for Banach algebras: The classical example is the disk algebra $A(\mathbb{D})$ (of those analytic functions on the open unit disk $\mathbb{D}\subset\mathbb{C}$ which extend continuously to the boundary circle $S^1$, equipped with the sup-norm) considered as a subalgebra of $C(S^1)$. The identity function $z\mapsto z$ is invertible in $C(S^1)$, but the inverse does not belong to $A(\mathbb{D})$.

Considering an $\ell^p$-Roe-like algebra as a subalgebra of the appropriate $\Blp$, one may ask whether the above property holds. The motivation to study this problem has its roots in the limit operator theory (see \cite{rabinovich2012limit} for the case of $\mathbb{Z}^N$), where this property plays an important role to study the Fredholmness of certain operators.

Using Theorem \ref{char for Roe alg Thm}, we can give an affirmative answer to the above question under some natural assumptions. More generally, we have the following result. The proof is similar to that of Theorem 5.1, ``(i) $\Leftrightarrow$ (ii)" in \cite{zhang2018}.
\begin{prop}
Let $X$ be a metric space with bounded geometry, and $p \in \Dom$. If $p\in (1,\infty)$, assume in addition that $X$ has Property A. Let $E$ be a Banach space and $B \subseteq \Blp$ a Banach subalgebra such that $C_b(X)BC_b(X)=B$ and $B$ is closed under block cutdowns. Let $a \in \Roe$. Then for any $b \in B$ such that $\Id-ab, \Id-ba\in\K$, we have $b\in\Roe$. In particular, if $B$ is closed under taking inverses, then so is $\Roe$.
\end{prop}

\begin{proof}
By Theorem \ref{char for Roe alg Thm} ``(iv) $\Rightarrow$ (i)", $[a,f]=0$ for all $f \in \VLinf$. Hence from Lemma \ref{commut replace}, for any $\epsilon>0$, there exists some $L_1>0$ such that $\|[a,f]\|<\epsilon/(3\|b\|^2)$ for any $L_1$-Lipschitz function $f \in C_b(X)_1$. On the other hand, since $K_1:=\Id-ab$ and $K_2:=\Id-ba$ are in $\K$, there exists some finite subset $F_0 \subseteq X$ such that
\begin{equation}\label{EQ18}
\|(\Id-P_{F_0})K_1\| < \frac{\epsilon}{12\|b\|} \andx \|K_2(\Id-P_{F_0})\| < \frac{\epsilon}{12\|b\|}.
\end{equation}
Choose a point $x_0 \in F_0$, and take $L_2=\epsilon/[6\mathrm{diam}F_0\cdot\|b\|\cdot(\|K_1\|+\|K_2\|)]$, $L:=\min\{L_1,L_2\}$.

For any $L$-Lipschitz function $f \in C_b(X)_1$, we take $\tilde{f}=f-f(x_0)$.  Then
\begin{equation}\label{EQ11}
[f,b]=[\tilde{f},b]=(ba+K_2)\tilde{f}b-b\tilde{f}(ab+K_1)=b[a,f]b-b\tilde{f}K_1+K_2\tilde{f}b.
\end{equation}
For the above $F_0 \subseteq X$, we have
\begin{equation}\label{EQ12}
b\tilde{f}K_1=b\tilde{f}(\Id-P_{F_0})K_1+bP_{F_0}\tilde{f}K_1, \andx K_2\tilde{f}b=K_2(\Id-P_{F_0})\tilde{f}b+K_2P_{F_0}\tilde{f}b.
\end{equation}
By (\ref{EQ18}), we obtain
\begin{equation}\label{EQ13}
\|b\tilde{f}(\Id-P_{F_0})K_1\| < \epsilon/6 \andx \|K_2(\Id-P_{F_0})\tilde{f}b\| < \epsilon/6.
\end{equation}
Furthermore, since $x_0 \in F_0$, we have
\begin{equation*}
\|P_{F_0}\tilde{f}\|\leq L_2\cdot\mathrm{diam}F_0=\frac{\epsilon}{6\|b\|\cdot(\|K_1\|+\|K_2\|)},
\end{equation*}
which implies that
\begin{equation}\label{EQ14}
\|bP_{F_0}\tilde{f}K_1\|\leq \epsilon/6, \andx \|K_2P_{F_0}\tilde{f}b\| \leq \epsilon/6.
\end{equation}
Hence combining (\ref{EQ11}), (\ref{EQ12}), (\ref{EQ13}) and (\ref{EQ14}), we have:
\begin{eqnarray*}
  \|[f,b]\| &\leq & \|b[a,f]b\| + \|b\tilde{f}K_1\| + \|K_2\tilde{f}b\| \\
   &\leq & \|b\|\cdot \frac{\epsilon}{3\|b\|^2} \cdot \|b\| + \|b\tilde{f}(\Id-P_{F_0})K_1\| + \|bP_{F_0}\tilde{f}K_1\|\\[0.3cm]
   && + \|K_2(\Id-P_{F_0})\tilde{f}b\| + \|K_2P_{F_0}\tilde{f}b\| \\[0.3cm]
   & \leq & \frac{\epsilon}{3} + \frac{\epsilon}{6} + \frac{\epsilon}{6} + \frac{\epsilon}{6} + \frac{\epsilon}{6} = \epsilon.
\end{eqnarray*}
By Lemma \ref{commut replace} again, we have that $[b,f]=0$ for all $f \in \VLinf$. Finally, applying Theorem \ref{char for Roe alg Thm} ``(i) $\Rightarrow$ (iv)", $b \in \Roe$.
\end{proof}

Consequently, we have the following corollary from Examples \ref{uniform Roe alg} and \ref{BD op}.
\begin{cor}
For the above $X$ and $p \in \Dom$, the uniform Roe algebra $B^p_u(X)$ and the band-dominated operator algebra $\mathcal{A}^p_E(X)$ are closed under taking inverses.
\end{cor}

\subsection{Questions}
We close the paper with two natural questions arising from the results in this paper.

First, note that our characterisation of Property A (Proposition \ref{prop:charact-of-A}, \eqref{item:char-of-A-eps} and \eqref{item:char-of-A-c}) relies on the result of Sako \cite{sako2014property} to get from the operator norm localisation property back to Property A. This result relies on positivity in an essential way, so this argument only work in the Hilbert space case, i.e. when $p=2$. However, Lemma \ref{ONL} works for any $p\in(1,\infty)$.
\begin{question}
Is it possible to generalise the characterisation of Property A (Proposition \ref{prop:charact-of-A}) to any $p\in[1,\infty)$?
\end{question}

Second, in the case of $p \in (1,\infty)$ or especially when $p=2$, is the assumption of Property A really necessary for all the quasi-local operators to be approximable by operators with finite propagation? In other words:
\begin{question}
Is there an example of a space $X$ with bounded geometry and a quasi-local operator on (say) $\B(\ell^2(X))$ which does not belong to the uniform Roe algebra $B^2_u(X)$ (usually denoted $C^*_u(X)$ in the literature)?
\end{question}

\bibliographystyle{plain}
\bibliography{bibQL}

\end{document}